\documentclass[a4paper,12pt]{article}
\usepackage[cp1251]{inputenc}
\usepackage{amsmath, amsthm, amsfonts, amssymb}

\usepackage{hyperref}

\usepackage{xcolor}
\usepackage{bbm}

\makeatletter
\@addtoreset{equation}{section}
\makeatother
\makeatletter
\@addtoreset{lem}{section}
\makeatother
\makeatletter
\@addtoreset{expl}{section}
\makeatother
\makeatletter
\@addtoreset{remk}{section}
\makeatother
\makeatletter
\@addtoreset{corl}{section}
\makeatother
\makeatletter
\@addtoreset{defn}{section}
\makeatother
\newcommand{\be}{\begin{equation}}
\newcommand{\ee}{\end{equation}}
\newcommand{\bel}{\begin{equation}\label}

\newcommand{\ba}{\begin{aligned}}
\newcommand{\ea}{\end{aligned}}

\newcommand{\mul}{\begin{multline}}
\newcommand{\emul}{\begin{multline}}

\newcommand{\eee}{{\rm e}}

\DeclareMathOperator{\1}{\mathbbm{1}}

\newcommand{\ve}{\varepsilon}

\newcommand{\wt}{\widetilde}

\newcommand{\mbN}{{\mathbb N}}
\newcommand{\mn}{{\mathbb N}}
\newcommand{\me}{{\mathbb E}}
\newcommand{\mmp}{\mathbb{P}}
\newcommand{\mbZ}{{\mathbb Z}}
\newcommand{\mbR}{{\mathbb R}}
\newcommand{\mr}{\mathbb{R}}

\renewcommand{\lg}{\langle}
\newcommand{\rg}{\rangle}
\newcommand{\Pb}{\mathbb{P}}

\newcommand{\sgn}{\mathop{\rm sgn}}

\newcommand{\toP}{\overset{\Pb}{\to}}

\theoremstyle{plain}
\newtheorem{thm}{Theorem}[section]
\newtheorem{lem}{Lemma}[section]
\newtheorem{prop}{Proposition}[section]

\theoremstyle{definition}

\newtheorem{expl}{Example}[section]
\theoremstyle{remark}
\newtheorem{remk}{Remark}[section]

\begin{document}
\title{On multidimensional locally perturbed standard random walks}
\date{}
\author{Congzao Dong\footnote{School of Mathematics and Statistics, Xidian University, Xi'an, P.R. China;\newline e-mail address: czdong@xidian.edu.cn}, \ \ Alexander Iksanov\footnote{Faculty of Computer Science and Cybernetics, Taras Shevchenko National University of Kyiv, Ukraine; e-mail address: iksan@univ.kiev.ua} \ \  and \ \ Andrey Pilipenko\footnote{Institute of Mathematics of Ukrainian National Academy of Sciences, Ukraine and
Igor Sikorsky Kyiv Polytechnic Institute, Ukraine; e-mail address: apilip@imath.kiev.ua}}

\maketitle

\begin{abstract}
Let $d$ be a positive integer and $A$ a set in $\mathbb{Z}^d$, which contains finitely many points with integer coordinates. We consider $X$ a standard random walk perturbed on the set $A$, that is, a Markov chain whose transition probabilities from the points outside $A$ coincide with those of a standard random walk on $\mathbb{Z}^d$, whereas the transition probabilities from the points inside $A$ are different. We investigate the impact of the perturbation on a scaling limit of $X$. It turns out that if $d\geq 2$, then in a typical situation the scaling limit of $X$ coincides with that of the underlying standard random walk. This is unlike the case $d=1$ in which the scaling limit of $X$ is usually a skew Brownian motion, a skew stable L\'{e}vy process or some other `skew' process. The distinction between the one-dimensional and the multidimensional cases under comparable assumptions may simply be caused by transience of the underlying standard random walk in $\mathbb{Z}^d$ for $d\geq 3$. More interestingly, in the situation where the standard random walk in $\mathbb{Z}^2$ is recurrent, the preservation of its Donsker scaling limit is secured by the fact that the number of visits of $X$ to the set $A$ is much smaller than in the one-dimensional case. As a consequence, the influence of the perturbation vanishes upon the scaling. On the other edge of the spectrum is the situation in which the standard random walk admits a Donsker's scaling limit, whereas its locally perturbed version does not because of huge jumps from the set $A$ which occur early enough.
\end{abstract}



\hphantom{2020 Mathematics Subject Classification: } Secondary: 60G50

\vskip15pt
\section{Introduction}\label{sec:Intro}
Let $d\in\mn$ and $\xi_1$, $\xi_2,\ldots$ be independent copies of a random variable $\xi$ taking values in $\mathbb{Z}^d$. Put $S_\xi(0):=(0,\ldots, 0)$ and $S_\xi(n):=\xi_1+\ldots+\xi_n$ for $n\in\mn$. The random sequence $S_\xi:=(S_\xi(n))_{n\in\mn_0}$ is called a {\it standard random walk}.

Let $d=1$. It is known that if $\xi$ has zero mean and finite variance, then Donsker's scaling of $S_\xi$ converges in distribution to a Brownian motion. Consider now a Markov chain whose transition probabilities coincide with those of $S_\xi$ everywhere except on a finite set, which is called a {\it membrane}. Donsker's scaling limit of this perturbed random walk is not necessarily a Brownian motion. The first result of this type was obtained by Harrison and Shepp \cite{Harrison+Shepp}. They investigated a one-dimensional simple symmetric random walk perturbed at $0$, that is, a Markov chain $X$, say, on $\mbZ$ with transition probabilities $p_{i,\,i\pm 1}=1/2$ for $i\neq 0$, $p_{0,\,1}=p$ and $p_{0,\,-1}=1-p$, where $p\in[0,1]$. Harrison and Shepp proved that $(n^{-1/2} X(\lfloor nt\rfloor))_{t\geq 0}$ converges in distribution as $n\to\infty$ to a skew Brownian motion, which is a solution to the stochastic differential equation
\[
{\rm d}Y(t)={\rm d}W(t)+\gamma {\rm d} L^Y_0(t),\quad t\geq 0.
\]
Here, $\gamma=2p-1$, $W$ is a standard Brownian motion, and $L^Y_0$ is a symmetric local time of $Y$ at $0$.

The result looks natural, for the number of returns to $0$ of a simple symmetric random walk normalized by $\sqrt n$ converges to a local time at $0$. An investigation of the scaling limit becomes
much more complicated in the situation that the jumps outside the membrane are not unit, but still have zero mean and finite variance. If the jumps from the membrane have a finite moment, then the scaling limit is still a skew Brownian motion as shown in \cite{Iksanov+Pilipenko:2016, MinlosZhizhina, NgoPeigne, PilipenkoPrykhodkoUMZh}. Non-trivial arguments are needed to prove the result and calculate the parameter $\gamma$. The problem becomes even more intriguing if the jumps outside the membrane have zero mean and finite variance, whereas the distribution of the jumps from the membrane belongs to the domain of attraction of an $\alpha$-stable distribution with $\alpha\in(0,1)$. On the one hand, the standard random walk generated by the (heavy-tailed) jumps from the membrane grows much faster then a standard random walk with jumps of finite mean. This leads to a guess that `Donsker's scaling limit does not exist because the perturbed random walk jumps too far at the scale $\sqrt{n}$'. On the other hand, a large jump from the membrane causes a large return time to the membrane. Hence, the number of jumps from the membrane cannot be too large. It was proved in \cite{IksanovPilipenkoPrykhodko2021, PilipekoPrykhodko2014jump_exit, PilipenkoSarantsev2023} that the latter reasoning is adequate, and that Donsker's scaling limit is a Brownian motion with a jump-type exit from $0$. It seems the result is no longer intuitive. Assume now that the distributions of the jumps outside the membrane and the jumps from the membrane belong to the domains of attraction of stable distributions with finite and infinite mean, respectively. Recently it was shown in \cite{DongIksanovPilipenko_2023_skewLevy} that a scaling limit of the corresponding is a skew stable L{\'e}vy process.

Multidimensional perturbed random walks have never been investigated in the same generality as the one-dimensional ones. The article \cite{BogdanskiPavlyukevichPilipenko_2022} considered the situation in which the jumps have a unit length, and a membrane is located at a hyperplane and has a periodic structure. The limit process is a multidimensional Brownian motion with a semipermeable membrane located at the hyperplane. The papers \cite{SzaszPaulin2010, SzaszTelcs81} were concerned with finite membranes. It was shown that if the jumps from the membrane admit a moment of some positive order, then the weak limit is a Brownian motion, that is, the perturbations do not affect the limit.

We generalize the latter result in several ways. We prove in Proposition \ref{thm:3_4_limit_transient} that if the underlying standard random walk is transient, then perturbations on a finite set have no effect at all. This situation occurs, for instance, for any genuine multidimensional perturbed random walk in dimensions greater then $2$. Motivated by this observation we consider the two-dimensional case and assume that the jumps outside the membrane have zero mean and finite variance. The underlying standard random walk is recurrent, but what is important now is that points are polar for a two-dimensional Brownian motion. As a consequence, the number of visits of the random walk to a finite membrane is relatively small. Hence, the perturbations should not affect the weak limit provided that their distribution tails are not too heavy. This reasoning is made precise in Theorem \ref{thm:3_4_limit_transient1} under the assumption that the distribution tails of perturbations exhibit a superlogarithmic decay. We believe that our argument is more probabilistic than the proofs given in   \cite{SzaszPaulin2010, SzaszTelcs81}. We also give in Theorem \ref{thm:negative result} a counterexample, which shows that the weak limit does not exist at all provided that the distribution tails of perturbations are too heavy. Now we explain why we deem the result quite unexpected. Observe that an effect of any fixed perturbation vanishes upon scaling by $\sqrt{n}$ and that a large jump from the membrane leads to a large return time to the membrane. Thus, extremely large jumps from the membrane should decrease the number of returns, hence an effect of all perturbations accumulated up to time $n$.

\section{Main results}\label{section:Limits of multidimensional perturbed RW}

Let $A$ be a given set in $\mathbb{Z}^d$, which contains finitely many points with integer coordinates. Denote by $X=(X(n))_{n\in\mn_0}$ a Markov chain in $\mbZ^d$ with the transition probabilities
\begin{equation}\label{eq:transition}
\Pb\{X(n+1)=y| X(n)=x\}:=
 \begin{cases}
\Pb\{\xi=y-x\},\quad~~~\text{if}~ x\notin A;\\
\Pb\{\eta^{(x)}=y-x\},\quad\text{if}~  x\in A,
  \end{cases}
\end{equation}
where $\eta^{(x)}$ for $x\in A$ are random variables taking values in $\mathbb{Z}^d$. We call $X$ a {\it standard random walk perturbed on the set $A$} or just a locally perturbed standard random walk.

As usual, $\lg \cdot,\cdot\rg$ will denote the scalar product in $\mr^d$, and $D([0,\infty),\mr^d)$ will denote the {\it Skorokhod space}, that is, the set of all c\`{a}dl\`{a}g functions which are defined on $[0,\infty)$ and take values in $\mr^d$. Throughout the paper we assume that the Skorokhod space is endowed with the $J_1$-topology. Comprehensive information on the $J_1$-topology can be found in the books \cite{Ethier-Kurtz, Whitt:2002}.

We start by discussing a trivial situation in which the asymptotic behavior of $X$ is driven by that of $S_\xi$, just because the set $A$ is hit by $X$ finitely often.
\begin{prop}\label{thm:3_4_limit_transient}
Let $d\in\mn$. Assume that, for a sequence of positive numbers $(a_n)_{n\in\mn}$ and a stochastic process $\mathcal{S}$,
\bel{eq:34_stable_lim}
\Big(\frac{S_\xi(\lfloor nt\rfloor)}{a_n}\Big)_{t\geq 0}~\Rightarrow~(\mathcal{S}(t))_{t\geq 0},\quad n\to\infty
\ee
in the $J_1$-topology on $D([0,\infty),\mr^d)$, and that the perturbed random walk $X$ visits $A$ finitely often with probability $1$.
Then $$\Big(\frac{X (\lfloor nt\rfloor)}{a_n}\Big)_{t\geq 0}~\Rightarrow~(\mathcal{S}(t))_{t\geq 0},\quad n\to\infty$$
in the $J_1$-topology on $D([0,\infty),\mr^d)$.
\end{prop}

Following p.~287 in \cite{Kesten+Spitzer:1963} or p.~20 in \cite{Spitzer:2001}, we call the random walk $S_\xi$ {\it aperiodic}, if no proper subgroup of $\mbZ^d$ contains all $x\in\mbZ^d$ satisfying $\mmp\{\xi=x\}>0$. By Theorem 1 on p.~67 in \cite{Spitzer:2001}, the walk $S_\xi$ is aperiodic if, and only if, the characteristic function $\psi$ defined by $\psi(z):=\me \eee^{{\rm i}\lg z, \xi\rg}$ for $z\in\mr^d$ possesses the following property: $\psi(z)=1$ if, and only if, each coordinate of $z$ is an integer multiple of $2\pi$. The class of aperiodic distributions on $\mr$ coincides with the class of {\it $1$-arithmetic distributions}, that is, the distributions which are concentrated on the integers and not concentrated on any sparser centered lattice.

We introduce the following condition.

\noindent {\sc Condition $\mathcal{B}$}: the random walk $S_\xi$ is aperiodic and, as far as the states of the Markov chain $X$ are concerned, all points (states) of $\mathbb{Z}^d\backslash A$ are accessible from any point of $A$.

Under Condition $\mathcal{B}$, if the random walk $S_\xi$ is transient, then so is the Markov chain $X$. Here are three examples of situations, in which, under Condition $\mathcal{B}$, Proposition \ref{thm:3_4_limit_transient} applies.
\begin{expl}
Let $d\geq 3$. Under the first part of Condition $\mathcal{B}$, $S_\xi$ is a genuinely $d$-dimensional standard random walk. Hence, it is transient by Theorem 1 on p.~83 in \cite{Spitzer:2001}.
\end{expl}

\begin{expl}
Let $d=1$ and the distribution of $\xi$ be $1$-arithmetic. Assume that the function $t\mapsto \mmp\{|\xi|>t\}$ is regularly varying at $\infty$ of index $-\alpha$ for some $\alpha\in (0,1)$ and that $$\mmp\{\xi>t\}\sim c_+ \mmp\{|\xi|>t\}\quad\text{and}\quad \mmp\{-\xi>t\}\sim c_- \mmp\{|\xi|>t\},\quad t\to\infty$$ for some nonnegative $c_+$ and $c_-$ summing up to $1$. Then \eqref{eq:34_stable_lim} holds with $a_n=b(n)$ for $b$ regularly varying at $\infty$ of index $1/\alpha$ and $\mathcal{S}=\mathcal{S}_\alpha$ an $\alpha$-stable L\'{e}vy process. By Gnedenko's local limit theorem (see, for instance, Theorem 8.4.1 in \cite{Bingham+Goldie+Teugels:1989}), $\lim_{n\to\infty} a_n\mmp\{S_n=0\}=g_\alpha(0)$, where $g_\alpha$ is the density of $\mathcal{S}_\alpha(1)$. Since the characteristic function of $\mathcal{S}_\alpha(1)$ is absolutely integrable $\big|\me \exp({\rm i}z\mathcal{S}_\alpha(1))\big|=\exp(-c_\alpha |z|^\alpha)$ for $z\in\mr$ and some constant $c_\alpha>0$, $g_\alpha$ is a bounded function which particularly entails that $g_\alpha(0)<\infty$. This demonstrates that $\sum_{n\geq 0}\mmp\{S_\xi(n)=0\}<\infty$, thereby proving transience of $S_\xi$.
\end{expl}

\begin{expl}
Let $d=2$ and $$\mmp\{\xi=x\}=\frac{c}{1+|x|^{2+\alpha}},\quad x\in\mathbb{Z}^2$$ for some $\alpha>0$ and a constant $c:=(\sum_{x\in\mathbb{Z}^2}(1+|x|^{2+\alpha})^{-1})^{-1}$. According to Section B.1 in \cite{Caputo+Faggionato+Caudilliere:2009}, the random walk $S_\xi$ is transient if, and only if, $\alpha\in (0,2)$. In the present situation condition \eqref{eq:34_stable_lim} holds with $a_n=n^{1/\alpha}$ and $\mathcal{S}$ being a rotation invariant two-dimensional $\alpha$-stable L\'{e}vy process (its characteristic function is given by $z\mapsto \eee^{-c_\alpha|z|^\alpha}$, $z\in\mr^2$ for appropriate positive constant $c_\alpha$).
\end{expl}

Our next theorem states that if a two-dimensional random walk $S_\xi$ is recurrent and the distribution tails of perturbations are not too heavy, then the scaling limit of $X$ coincides with that of $S_\xi$ and, as such, is not affected by the presence of perturbations.
\begin{thm}\label{thm:3_4_limit_transient1}
Assume that $d=2$, $\me [\xi]=(0,0)$, $\me[|\xi|^2]<\infty$, Condition $\mathcal{B}$ holds and
\bel{eq:35_62}
\max_{x\in A}\Pb\big\{\big|\eta^{(x)}\big|>t\big\}=o\Big(\frac{1}{\log t}\Big),\quad t\to\infty.
\ee
Then $$\Big(\frac{X(\lfloor nt\rfloor)}{\sqrt{n}}\Big)_{t\geq 0}~\Rightarrow~ \Big(W_\Gamma(t)\Big)_{t\geq 0},\quad n\to\infty$$
in the $J_1$-topology on $D([0,\infty), \mr^2)$. Here,
$W_\Gamma:=(W_\Gamma(t))$ is a two-dimensional Wiener process with the characteristic function
$$\me \eee^{{\rm i}\lg z, W_\Gamma(t)\rg}=\eee^{-\frac{\lg \Gamma z,z\rg t}{2}},\quad z\in\mbR^2$$ and $\Gamma$ is a nondegenerate covariance matrix of $\xi$.
\end{thm}
\begin{remk}
Assume that $\me [\xi]=(0,0)$, $\me [|\xi|^2]<\infty$, and the matrix $\Gamma$ is degenerate. This is equivalent to the fact that the coordinates $\xi^{(1)}$ and $\xi^{(2)}$ of $\xi=(\xi^{(1)}, \xi^{(2)})$ satisfy $\xi^{(1)}+c\xi^{(2)}=0$ a.s.\ for some $c\in\mathbb{Z}$ or $\xi^{(2)}=0$ a.s. As a consequence, the random walk $S_\xi$ is periodic. In particular, condition $\mathcal{B}$  does not hold. Assume that $c\neq 0$ is integer and that $\xi^{(1)}$ has a nondegenerate $1$-arithmetic distribution with $\sigma^2:=\me [(\xi^{(1)})^2]<\infty$. The behavior of $S_\xi$ is then essentially $1$-dimensional, for the coordinates of $\xi$ are multiples of each other. In particular, by Donsker's theorem,
$$\Big(\frac{S_\xi(\lfloor nt\rfloor)}{\sigma \sqrt{n}}\Big)_{t\geq 0}~\Rightarrow~ (W, c^{-1}W),\quad n\to\infty$$
in the $J_1$-topology on $D([0,\infty), \mr^2)$, where $W$ is a standard Brownian motion.

For instance, assume that $A=\{(0,0)\}$,  $\xi^{(1)}=-\xi^{(2)}\in\{-1,0,1\}$ a.s.,  $\eta^{((0,0))}=(\eta, -\eta)$ for some integer-valued $\eta$ with $\me [|\eta|]<\infty$, and $X(0)=(0,0)$. Then $c=-1$ and Theorem 1.1 in \cite{Iksanov+Pilipenko:2016} implies that $$\Big(\frac{X(\lfloor nt\rfloor)}{\sigma \sqrt{n}}\Big)_{t\geq 0}~\Rightarrow~ (W^{{\rm skew}}_\gamma,-W_\gamma^{{\rm skew}}),\quad n\to\infty$$ in the $J_1$-topology on $D([0,\infty), \mr^2)$, where $W^{{\rm skew}}_\gamma:=(W^{{\rm skew}}_\gamma(t))_{t\geq 0}$ is a skew Brownian motion with a permeability  parameter $\gamma=\me [\eta]/\me [|\eta|]$, that is, a Markov process on $\mbR$ with the transition probability density function
\begin{equation*}
p_t(x,y) = \varphi_t (x-y) + \gamma\sgn(y)\varphi_t (|x|+|y|),\quad x,y \in \mbR,\quad t>0
\end{equation*}
and $\varphi_t(z)=(2\pi t)^{-1/2}\eee^{-z^2/(2t)}$ for $z\in\mbR$.
\end{remk}

Skorokhod in \cite{Skorokhod:1956} introduced the $J_1$, $J_2$, $M_1$ and $M_2$- topologies on $D([0,\infty),\mr^d)$. Nowadays these are known as the Skorokhod topologies, with the $J_1$-topology being the most widely used one. In our last main result, Theorem \ref{thm:negative result}, we treat the situation in which the distribution tail of perturbations is extremely heavy. As a consequence, the distributions of the Donsker scaling of $X$ are not tight in any of the four Skorokhod topologies. Heuristically, this is caused by the fact that, with a probability bounded away from $0$, there is a huge jump from $A$, which occurs early enough.
\begin{thm}\label{thm:negative result}
Let $X$ be a simple random walk in $\mbZ^2$ perturbed on the set $A=\{(0,0)\}$ with $X(0)=(0,0)$ a.s.\ and probability $1/4$ of moving from any point of $\mbZ^2\setminus\{(0,0)\}$ to its closest neighbor in $\mbZ^2$. Assume that, for a constant $a>0$,
\bel{eq:2_4_tail_llt}
\Pb\{|X(1)|>t\}~\sim~ \frac{a}{\log \log t},\quad t\to \infty.
\ee
Then
\bel{eq:2_4_conv_abs_infty}
\liminf_{n\to\infty}\Pb\{\max_{1\leq k\leq n}\,|X(k)|>n\}>0.
\ee

In particular, the sequence of distributions of $(n^{-1} X(\lfloor nt\rfloor))_{t\geq 0}$ is not tight in any of the Skorokhod topologies on $D([0,\infty), \mr^2)$.
\end{thm}

\section{A representation for $X$}
In what follows, for typographical ease we always write $0$ for zero vectors. 

For each $x\in A$, let $\eta^{(x)}_1$, $\eta^{(x)}_2,\ldots$ be independent copies of a random variable $\eta^{(x)}$, which are also independent of $\xi_1$, $\xi_2,\ldots$ It is assumed that the variables in the collections corresponding to different $x\in A$ are also independent. Put $S_{\eta^{(x)}}(0):=0$ and $S_{\eta^{(x)}}(n):=\eta^{(x)}_1+\ldots+\eta^{(x)}_n$ for $n\in\mn$ and $x\in A$.

Without loss of generality we assume that $X$ admits a representation
\bel{eq:20_representation_PRW}
X(n)=X(0)+S_\xi(n-T(n))+\sum_{x\in A} S_{\eta^{(x)}}(T^{(x)}(n)),\quad n\in\mn\quad \text{a.s.},
\ee
where, for $n\in\mn$,
\[
T^{(x)}(n):=\sum_{k=0}^{n-1} \1_{\{X(k)=x\}}
\]
is the number of visits of $X$ to $x\in A$ up to and including time $n-1$, and
\[
T(n):=\sum_{x\in A}T^{(x)}(n)=\sum_{k=0}^{n-1}\1_{\{X(k)\in A\}}
\]
is the number of visits of $X$ to $A$, again, up to and including time $n-1$. The advantage of using such a coupled version is obvious: the sequence $X$ is now constructed pathwise, rather than just distributionally.

We only explain informally that the so defined $X$ has the transition probabilities as in \eqref{eq:transition} under the assumption $x\in A$. The argument is analogous in the complementary case $x\notin A$. Observe that $T^{(x)}(n+1)=T^{(x)}(n)+1$, $T(n+1)=T(n)+1$ and $T^{(z)}(n+1)=T^{(z)}(n)$ for $z\in A$, $z\neq x$. This entails
\begin{multline*}
X(n+1)=X(0)+S_\xi(n+1-T(n+1))+\sum_{z\in A,\ z\neq x} S_{\eta^{(z)}}(T^{(z)}(n+1))+S_{\eta^{(x)}}(T^{(x)}(n+1))\\=X(0)+S_\xi(n-T(n))+\sum_{z\in A,\ z\neq x} S_{\eta^{(z)}}(T^{(z)}(n))+S_{\eta^{(x)}}(T^{(x)}(n)+1)=X(n)+\eta^{(x)}_{T^{(x)}(n)+1}
\end{multline*}
and thereupon $$\Pb\{X(n+1)=y | X(n)=x\}=\Pb\{\eta^{(x)}=y-x\}$$ because $\eta^{(x)}_{T^{(x)}(n)+1}$ has the same distribution as $\eta^{(x)}$ and is independent of $X(n)$.

\section{Proof of Proposition \ref{thm:3_4_limit_transient}}

We start with an auxiliary result. As usual, $\overset{\Pb}\to$ will denote convergence in probability.
\begin{lem}\label{lem:simple}
Suppose \eqref{eq:34_stable_lim} and $T(n)/n \overset{\Pb}\to 0$ as $n\to\infty$. Then
\begin{equation}\label{eq:simple}
\left(\frac{S_\xi(\lfloor nt\rfloor-T(\lfloor nt\rfloor))}{a_n}\right)_{t\geq 0}~\Rightarrow~\Big(\mathcal{S}(t)\Big)_{t\geq 0},\quad n\to\infty
\end{equation}
in the $J_1$-topology on $D([0,\infty), \mr^d)$.
\end{lem}
\begin{proof}
Using the a.s.\ monotonicity of $T$ yields, for any $t_0>0$ and $n\in\mn$,
$$\sup_{t\in [0,t_0]}\Big|t-\frac{\lfloor nt\rfloor-T(\lfloor nt\rfloor)}{n}\Big|=\sup_{t\in [0,t_0]} \Big|\Big(t-\frac{\lfloor nt\rfloor}{n}\Big)+\frac{T(\lfloor nt \rfloor)}{n} \Big|\leq \frac{1}{n}+\frac{T(\lfloor nt_0\rfloor)}{n}\quad\text{a.s.}
$$
Therefore,
$$
\Big(\frac{\lfloor nt\rfloor-T(\lfloor nt\rfloor)}{n}\Big)_{t\geq 0}~\Rightarrow~ \Big(I(t)\Big)_{t\geq 0},\quad n\to\infty
$$
in the $J_1$-topology on $D([0, \infty),\mr)$, where $I(t):=t$ for $t\geq 0$.
The limit function $I$ is deterministic, continuous and strictly increasing. Hence, by Theorem 13.2.2 in \cite{Whitt:2002}, the last limit relation in combination with \eqref{eq:34_stable_lim} ensures that
\[
\Big(\frac{S_\xi (\lfloor nt\rfloor-T(\lfloor nt\rfloor))}{a_n}\Big)_{t\geq 0}~\Rightarrow~\Big(\mathcal{S}(t)\Big)_{t\geq 0}, \ n\to\infty,
\]
in the $J_1$-topology on $D([0,\infty), \mr^d)$.
\end{proof}

With this at hand we are ready to prove Proposition \ref{thm:3_4_limit_transient}.
\begin{proof}[Proof of Proposition \ref{thm:3_4_limit_transient}]
By assumption,$\sum_{x\in A} T^{(x)}(\infty):=\sum_{x\in A}\sum_{k\geq 0}\1_{\{X(k)=x\}}<\infty$ a.s. First, this trivially entails $T(n)/n \overset{\Pb}\to 0$ as $n\to\infty$ and thereupon \eqref{eq:simple} by Lemma \ref{lem:simple}. Second, this implies that
\begin{equation*}
\sum_{x\in A}S_{|\eta^{(x)}|}(T^{(x)}(\infty))<\infty\quad\text{a.s.}
\end{equation*}
In view of \eqref{eq:20_representation_PRW}
\bel{eq:34_35}
\sup_{t\in [0,t_0]}\Big|\frac{X (\lfloor nt\rfloor)}{a_n}-\frac{S_\xi (\lfloor nt\rfloor-T(\lfloor nt\rfloor))}{a_n}\Big|\leq
\sum_{x\in A}\frac{S_{|\eta^{(x)}|}(T^{(x)}(\infty))}{a_n}~\toP 0~,\quad n\to\infty,
\ee
which completes the proof.
\end{proof}

\section{Proof of Theorem \ref{thm:3_4_limit_transient1}}

Our proof of Theorem \ref{thm:3_4_limit_transient1} is essentially based on the result given next.
\begin{prop}\label{lem:number visits multidim_log}
Under the assumptions of Theorem \ref{thm:3_4_limit_transient1}, the sequence $(T(n)/\log n)_{n\geq 2}$ is bounded in probability.
\end{prop}
For a non-empty set $B\subset\mathbb{Z}^2$, put $\tau_B:=\inf\{k\in\mn: S_\xi(k)\in B\}$ with the usual convention that the infimum of the empty set is equal to $+\infty$. We shall write $\tau_0$ for $\tau_{\{0\}}$. The proof of Proposition \ref{lem:number visits multidim_log}, in its turn, relies on the following fact.
\begin{lem}\label{lem:2_4_123}
Under the assumptions of Theorem \ref{thm:3_4_limit_transient1},
\[\Pb\{\tau_0>n\}~\sim~\frac{c}{2\pi \sqrt{{\rm det}\,\Gamma}\, \log n},\quad n\to\infty,
\]
where $c$ is the greatest common divisor of the set $\{n\in\mn: \mmp\{S_\xi(n)=0\}>0\}$.
\end{lem}
\begin{proof}
Recall that the aperiodic random walk $S_\xi$ with $c=1$ is called strongly aperiodic. Under the additional assumption that the walk $S_\xi$ is {\it strongly aperiodic}, the result can be found in Lemma 2.1 of \cite{Bohun+Marynych:2021}. Although it is likely the claim is known in full generality, we have been unable to locate its complete form in the literature.

Observe that the random walk $(S_\xi(cn))_{n\in\mn_0}$, with jumps having the distribution of $S_\xi(c)$, is strongly aperiodic. By a local limit theorem for strongly aperiodic random walks (for instance, Proposition 9 on p.~75 in \cite{Spitzer:2001}),
\begin{equation}\label{eq:local}
\mmp\{S_\xi(cn)=0\}~\sim~ \frac{c}{2\pi \sqrt{{\rm det}\,\Gamma}}\; \frac{1}{n},\quad n\to\infty.
\end{equation}
The remainder of the proof is analogous to the proof given in Example 1 on p.~167 in \cite{Spitzer:2001} for simple random walks in the plane. We provide a sketch for completeness.

In view of \eqref{eq:local}, $$\sum_{k=0}^{cn}\mmp\{S_\xi(ck)=0\}~\sim~ \frac{c}{2\pi \sqrt{{\rm det}\,\Gamma}}\log n,\quad n\to\infty.$$ For $n\in\mn_0$, put $U_n:=\mmp\{S_\xi(cn)=0\}$ and $R_n:=\mmp\{\tau_0>n\}$. Using $\sum_{k=0}^{cn}U_kR_{cn-k}=1$ for $n\in\mn_0$, we infer, for any nonnegative integer $\ell\leq n$, $$R_{cn-c\ell}(U_0+\ldots+U_{c\ell})+U_{c\ell+1}+\ldots+U_{cn}\geq 1.$$ Choosing $\ell:=\ell(n)=n-\lfloor n/\log n\rfloor$ we obtain $$U_0+\ldots+U_{c\ell}~\sim~\frac{c}{2\pi \sqrt{{\rm det}\Gamma}}\log \ell~\sim~\frac{c}{2\pi \sqrt{{\rm det}\Gamma}}\log(n-\ell),\quad n\to\infty$$ and $$U_{c\ell+1}+\ldots+U_{cn}=O(1/\log n)~\to~0,\quad n\to\infty.$$ Hence, given $\varepsilon\in (0,1)$, $$\frac{c}{2\pi\sqrt{{\rm det}\,\Gamma}}\log (cn-c\ell)R_{cn-c\ell}\geq 1-\varepsilon$$ for large enough $n$ and thereupon $$\liminf_{n\to\infty}(\log n) R_{cn}\geq \frac{2\pi\sqrt{{\rm det}\,\Gamma}}{c}.$$ On the other hand, $$\limsup_{n\to\infty}(\log n) R_{cn}\leq \frac{2\pi\sqrt{{\rm det}\,\Gamma}}{c}$$ follows from $$1=\sum_{k=0}^{cn} U_kR_{cn-k}\geq R_{cn}(U_0+\ldots+U_{cn}).$$ Thus, $$R_{cn}~\sim~\frac{c}{2\pi \sqrt{{\rm det}\,\Gamma}\log n},\quad n\to\infty.$$
Monotonicity of  $(R_k)_{k\in\mn_0}$ enables us to replace $cn$ with $n$ and thereby secures the claim.
\end{proof}

We are ready to prove Proposition \ref{lem:number visits multidim_log}.
\begin{proof}[Proof of Proposition \ref{lem:number visits multidim_log}]
Any aperiodic two-dimensional random walk is genuinely two-dimensional. Recall that $\me [\xi]=0$ and $\me[|\xi|^2]<\infty$. Hence, by Theorem T1 on p.~83 in \cite{Spitzer:2001}, the random walk $S_\xi$ is recurrent. This taken together with the second part of Condition $\mathcal{B}$ ensures that the Markov chain $X$ is also recurrent.

Fix any $v\notin A$. We first prove the claim for an auxiliary Markov chain $\wt X:=(\wt X(n))_{n\in\mn_0}$ with $\wt X(0)$ having the same distribution as $X(0)$. The transition probabilities of $\wt X$ are given by
\bel{eq:aux_MC_e}
\Pb\{\wt X(1)= y\ |\ \wt X(0)=x\}=
\begin{cases}
\Pb\{ X(1)= y\ |\ X(0)=x\}= \Pb\{\xi=y-x\},& \text{if } x\notin A,\\
\1_{\{y=v\}},& \text{if } x\in A.
\end{cases}
\ee
Thus, jumps of $\wt X$ from the points outside $A$ have the same distribution as jumps of $X$, that is, the distribution of $\xi$, and $\wt X$ jumps to the fixed state $v$ upon hitting $A$.

In addition to the notation $\tau_A$ introduced earlier, put $\tau_A^{\wt X}:=\inf\{k\in\mn: \wt X(k)\in A\}$. Plainly, the Markov chain $\wt X$ is recurrent. This entails $\tau_A^{\wt X}<\infty$ a.s. By the Kesten-Spitzer ratio theorem (Theorem 4a in \cite{Kesten+Spitzer:1963}), for any $y\in \mbZ^2,$
\bel{eq:2_4_124}
\lim_{n\to\infty}\frac{\Pb\{\tau_{A-y}>n \}}{\Pb\{\tau_0>n\}}=g_A(y),
\ee
where $g_A(y):=\lim_{z\to\infty}\sum_{n\geq 0}\mmp\{S_\xi(n)=z, \tau_A>n|S_\xi(0)=y\}\in [0,\infty)$. As a consequence,
\[
\lim_{n\to\infty}\frac{\Pb\{\tau^{\wt X}_{A}>n\ |\  {\wt X}(0)=y\}}{\Pb\{\tau_0>n \}}=g_A(y),\quad y\notin A.
\]
This in combination with Lemma  \ref{lem:number visits multidim_log} yields
\bel{eq:2_4_129}
\Pb\{\tau^{\wt X}_{A}>n\ |\  {\wt X}(0)=y\}~\sim~\frac{cg_A(y)}{2\pi \sqrt{{\rm det}\,\Gamma}\log n},\quad n\to\infty.
\ee

Let $\rho$ be a random variable with distribution being the conditional distribution of $(\tau^{\wt X}_{A}+1)$ given ${\wt X}(0)=v$. Put $\wt N(t):=\sum_{k\geq 1}\1_{\{S_\rho(k)\leq t\}}$ for $t\geq 0$. The process $(\wt N(t))_{t\geq 0}$ is a renewal process that corresponds to the one-dimensional standard random walk $S_\rho$ with jumps having a slowly varying tail.  As a consequence of the functional weak convergence proved in Theorem 2.1 of \cite{Kabluchko+Marynych:2016}, $\wt N(n)/\log n$ converges in distribution as $n\to\infty$ to an exponentially distributed random variable ${\rm Exp}_v$ with mean $2\pi \sqrt{{\rm det}\Gamma}/(cg_A(v))$ if $g_A(v)>0$.
For $n\in\mn_0$, put $\wt T(n):=\sum_{k=0}^n \1_{\{\wt X(k) \in A\}}$, so that $\wt T(n)$ is the number of visits of $\wt X$ to $A$ up to and including time $n$. If $\wt X(0)\in A$, then $\wt T(n)$ has the same distribution as $1+\wt N(n)$. If $\wt     X(0)=y\notin A$, then $\wt T(n)$ has the same distribution as $(1+\wt N(n-\kappa))\1_{\{\kappa\leq n\}}$, where $\kappa$ is a random variable with distribution being the conditional distribution of $\tau^{\wt X}_{A}$ given ${\wt X}(0)=y$. It is assumed that $\kappa$ is independent of $S_\rho$. We infer that in both cases $\wt T(n)/\log n$ converges in distribution to ${\rm Exp}_v$ as $n\to\infty$. The proof of the claim for $\wt X$ is complete if $g_A(v)>0$.

We stress that there is no guarantee that the inequality $g_A(v)>0$ holds true for all $v\notin A$. For instance, if $A=\{(0,1), (0,-1), (1, 0), (-1,0)\}$, $v=(0,0)$ and $|\xi|=1$ a.s., then $\tau_{A}=1$ given $S_\xi(0)=(0,0)$, that is, $g_A(v)=0$. Nevertheless, we shall show that {\it there exists} $v\notin A$ such that $g_A(v)>0$.

Let $m\in\mn$ be the minimal number such that $A\subseteq [-m,m]^2\cap \mbZ^2$. For the proofs of both Theorem \ref{thm:3_4_limit_transient1} and Proposition \ref{lem:number visits multidim_log}, we put $\eta^{(x)}:=\xi$ for $x\in ([-m,m]^2\cap \mbZ^2)\setminus A$. This enables us to work with the perturbing set $A:=[-m,m]^2\cap \mbZ^2$ rather than original $A$. The advantage of this choice is justified by the result given next.
\begin{lem}\label{lem:positive_g_A}
If
\bel{eq:AisSQ}
A=[-m,m]^2\cap \mbZ^2,
\ee
then $g_A(v)>0$ for any $v\notin A$.
\end{lem}
\begin{proof}
Let $x, y\notin A$. In view of \eqref{eq:AisSQ}, there exist $k \in \mn$ and a path of length $k$ from $x$ to $y$ that does not visit $A$ and has a positive probability. Hence, for some $c>0$,
\[
  {\Pb\{ \tau_{A-x }>n+ k \}}\geq
c {\Pb\{ \tau_{A-y}>n\}}
\]
and consequently, for any $r>0$, there exists $k=k_r$ satisfying
\bel{eq:ratio_prob_pathA}
c_r:=  \inf_{|x-y|\leq r, \ x, y\notin A}\frac{\Pb\{ \tau_{A-x }>n+ k_r\}}{  \Pb\{ \tau_{A-y}>n\}}>0.
\ee
By Lemma \ref{lem:2_4_123}, $\lim_{n\to\infty}\frac{\Pb\{\tau_0>n +k\}}{\Pb\{\tau_0>n\}}=1$. This in combination with \eqref{eq:2_4_124} and \eqref{eq:ratio_prob_pathA} entails that either $g_A(x)>0 $ for all $ x\notin A$ or $g_A(x)=0$ for all $ x\notin A$. We shall prove that the first alternative prevails.

According to formula (1.34) in \cite{Kesten+Spitzer:1963},
\[
\sum_{z\in A}g_A(z)=1.
\]
Hence,
\begin{multline*}
1=\lim_{n\to\infty}\frac{\sum_{z\in A}\sum_{y\in\mbZ}\Pb\{ \xi=y\}\Pb\{\sigma_{A-(z+y)}>n-1\}}{\Pb\{\tau_0>n\}}\\=
\lim_{n\to\infty}\frac{\sum_{z\in A}\sum_{y\in\mbZ,\ y+z\notin A}\Pb\{ \xi=y\}\Pb\{\tau_{A-(z+y)}>n-1\}}{\Pb\{\tau_0>n\}},
\end{multline*}
where $\sigma_{B}:=\inf\{n\geq 0\ :\ S_\xi(n)\in B\}$. If the variable $\xi$ takes finitely many values $y_1,\ldots, y_\ell$, say, then the last formula ensures that
there exist $y\in \{y_1,\ldots,y_\ell\}$ and $z\in A$ satisfying $y+z\notin A$ and $\lim_{n\to\infty}\frac{  \Pb\{\tau_{A-(z+y)}>n-1\}}{\Pb\{\tau_0>n\}}>0$. This proves
$g_A(z+y)>0$ and thereupon $g_A(x)>0$ for all $x\notin A$.

Assume now that the variable $\xi$ takes infinitely many values. To complete the proof of the lemma, it suffices to show that, for some $x\notin A$, $k\in\mbN$ and any $z\in A$,
\bel{eq:ratio_tau}
\liminf_{n\to\infty}\frac{\Pb\{\tau_{A-x}>n+k \}}{\Pb\{\tau_{A-z}>n\}}>0.
\ee
For brevity, we shall only prove \eqref{eq:ratio_tau} for $z=0$. For $y=(y_1, y_2)\in\mbZ^2$, put $|y|:=\max(|y_1|, |y_2|)$. For any $n\geq 2$ and any fixed $x\in\mbZ^2$ with $|x|\geq 1$,
\begin{multline*}
\Pb\{\tau_{A}>n\}=\sum_{|y|>m}\Pb\{ \xi=y\}\Pb\{\tau_{A-y}>n-1\} =\\
\sum_{ m<|y|\leq m+ |x|}\Pb\{ \xi=y\}\Pb\{\tau_{A-y}>n-1\}+  \sum_{ |y|>m+|x|}\Pb\{ \xi=y\}\Pb\{\tau_{A-y}>n-1\}  \leq\\
\sum_{ m<|y|\leq m+ |x|} \Pb\{\tau_{A-y}>n-1\}+ c_{|x|}^{-1} \sum_{ |y|>m+|x|}\Pb\{ \xi=y\}\Pb\{\tau_{A-(y+x))}>n-1+k_{|x|}\}
\end{multline*}
having utilized \eqref{eq:ratio_prob_pathA} for the last inequality. Here, $k_{|x|}$ and $c_{|x|}$ are as in \eqref{eq:ratio_prob_pathA}. Another application of
\eqref{eq:ratio_prob_pathA} to the first summand on the right-hand side of the last inequality yields
\begin{multline*}
\Pb\{\tau_{A}>n\}\leq c_{m+2|x|}^{-1}\sum_{ m<|y|\leq m+ |x|} \Pb\{\tau_{A-x}>n-1+k_{m+2|x|}\}\\+ c_{|x|}^{-1} \sum_{ |y|>m+|x|}\Pb\{ \xi=y\}\Pb\{\tau_{A-(y+x))}>n-1+k_{|x|}\}\\\leq
C_1\Big(\Pb\{\tau_{A-x}>n-1+k_{m+2|x|}\}+  \sum_{ y+x\notin A}\Pb\{ \xi=y\}\Pb\{\tau_{A-(y+x)}>n-1+k_{|x|}\}\Big)\\=C_1\big(\Pb\{\tau_{A-x}>n-1+k_{m+2|x|}\}+
\Pb\{\tau_{A-x}>n+k_{|x|}\}\big)\leq C_2  \Pb\{\tau_{A-x}>n-1+K\},
\end{multline*}
where $C_1$, $C_2$ and $K$ are some positive constants. This proves \eqref{eq:ratio_tau}. Hence, $g_A(x)>0$ for any $x\notin A$.

The proof of Lemma \ref{lem:positive_g_A} is complete.
\end{proof}

Now we continue the proof of Proposition \ref{lem:number visits multidim_log}. Condition $\mathcal{B}$ implies that, for some $u\in A$ and $v\notin A$,
\[
\Pb\Big\{\bigcup_{k\geq 1}\{X(1)\notin A,\ldots, X(k-1)\notin A, X(k)=u\ |\ X(0)=v\}\Big\}>0.
\]
Invoking condition $\mathcal{B}$ once again we conclude that there exist
$l\in\mn$ and distinct elements $v_1,\ldots, v_{l-1}\in A$ other than $u$ and $v$ (the collection is empty if $l=1$) such that
\[
\Pb\{X(1)=v_1,\ldots,X(l-1)=v_{l-1}, X(l)=v\ |\ X(0)=u\}>0.
\]
We claim that, for all $c>0$ and all $n\in \mbN$,
\bel{eq:2_4_123}
\Pb\Big\{\sum_{k=0}^n \1_{\{X(k)=u, X(k+1)=v_1,\dots, X(k+l-1)=v_{l-1}, X(k+l)=v\}}>c\Big\}\leq \Pb\Big\{\sum_{k=0}^n\1_{\{\wt X(k)\in A\}}>c\Big\}.
\ee
To prove \eqref{eq:2_4_123}, put
$$
\wt \sigma_0:=0,\ \ \sigma_1:=\inf\{i\in\mn_0: X(i)\in A\},$$
$$
\wt \sigma_k:=\inf\{i>\sigma_k: X(i)=v\},\quad \sigma_{k+1}:=\inf\{i>\wt\sigma_k: X(i)\in A\},\quad k\in\mn,
$$
\[
\lambda(n):=\sum_{k\geq 0}\sum_{i=0}^{n}\1_{\{\wt \sigma_k\leq i \leq \sigma_{k+1}\}}, \quad n\in\mn_0.
\]
and
\[
\lambda^{\leftarrow}(n):= \inf\{ k\geq 0\ :\ \lambda(k)>n\},\quad n\in\mn_0.
\]
Since the sequence $(X(\lambda^{\leftarrow}(n)))_{n\geq 0}$ has the same distribution as
$(\wt X(n))_{n\geq 0}$ and
$$
\sum_{k=0}^n\1_{\{X(k)=u, X(k+1)=v_1,\ldots, X(k+l-1)=v_{l-1}, X(k+l)=v\}}\leq \sum_{k=0}^n\1_{\{X(\lambda^{\leftarrow}(k))\in A\}}\quad\text{a.s.},$$
\eqref{eq:2_4_123} follows. Using now \eqref{eq:2_4_123} together with the already proved claim for $\wt X$ we conclude that the sequence
\begin{equation}\label{eq:bounded}
\Big(\frac{\sum_{k=1}^n\1_{\{X(k)=u, X(k+1)=v_1,\ldots, X(k+l-1)=v_{l-1}, X(k+l)=v\}}}{\log n}\Big)_{n\geq 2}~~\text{is bounded in probability}.
\end{equation}

Put $\zeta_0:=0$ and
$$
 \zeta_{k+1}=\inf\{i>\zeta_k: X(i)\in A\},\quad k\in\mn.
$$
Since the Markov chain $X$ is recurrent, the random variables $\zeta_1$, $\zeta_2,\ldots$ are a.s.\ finite. The sequence $(X(\zeta_k))_{k\in\mn}$ is a Markov chain taking values in $A$. Condition $\mathcal{B}$ ensures that $(X(\zeta_k))_{k\in\mn}$ admits a unique stationary distribution $(\pi_x)_{x\in A}$ with $\pi_u>0$. Hence, by the strong law of large numbers for Markov chains,
\[
\lim_{n\to\infty} \frac{\sum_{k=1}^n\1_{\{X(k)=u\}}}{\sum_{k=1}^n\1_{\{X(k)\in A\}}}=
\lim_{n\to\infty} \frac{\sum_{k=1}^n\1_{\{X(\zeta_k)=u\}}}{n}=\pi_u\quad \text{a.s.}
\]
and thereupon
\begin{multline*}
\lim_{n\to\infty} \frac{ \sum_{k=1}^n\1_{\{X(k)=u, X(k+1)=v_1,\dots, X(k+l-1)=v_{l-1}, X(k+l)=v\}}}{\sum_{k=1}^n\1_{\{X(k)\in A\}}}=\\
\pi_u\, \Pb\{X(1)=v_1,\ldots,X(l-1)=v_{l-1}, X(l)=v\ |\ X(0)=u\}>0\quad \text{a.s.}
\end{multline*}
This in combination with \eqref{eq:bounded} completes the proof of Proposition \ref{lem:number visits multidim_log}.
\end{proof}

\begin{proof}[Proof of Theorem \ref{thm:3_4_limit_transient1}]
Similarly to \eqref{eq:34_35}, for any $t_0>0$ and $n\in\mn$,
$$\sup_{t\in[0,\,t_0]}\,\Big|\frac{X (\lfloor nt\rfloor )}{\sqrt{n}}-\frac{S_\xi (\lfloor nt\rfloor-T(\lfloor nt\rfloor))}{\sqrt{n}}\Big| \leq \frac{X(0)}{\sqrt{n}}+
\sum_{x\in A}\sum_{k=1}^{T^{(x)}(n t_0)} \frac{|\eta^{(x)}_k|}{\sqrt{n}}\quad \text{a.s.}$$
By a functional limit theorem for multidimensional standard random walks (see, for instance, Theorem 4.3.5 in \cite{Whitt:2002}),
\[
\Big(\frac{S_\xi(\lfloor nt\rfloor)}{\sqrt{n}}\Big)_{t\geq 0}~\Rightarrow~ \Big(W_\Gamma(t)\Big)_{t\geq 0},\quad n\to\infty
\]
in the $J_1$-topology on $D([0,\infty), \mr^2)$. Proposition \ref{lem:number visits multidim_log} entails $$\frac{T(n)}{n}~\overset{\Pb}\to~ 0, \quad n\to\infty.$$ Invoking Lemma \ref{lem:simple} we infer $$\Big(\frac{S_\xi (\lfloor nt\rfloor-T(\lfloor nt\rfloor))}{\sqrt{n}}\Big)_{t\geq 0}~\Rightarrow~ \Big(W_\Gamma(t)\Big)_{t\geq 0},\quad n\to\infty$$
 in the $J_1$-topology on $D([0,\infty), \mr^2)$. It remains to prove that
$$\sum_{x\in A}\sum_{k=1}^{T^{(x)}(n t_0)} \frac{|\eta^{(x)}_k|}{\sqrt{n}}~\toP~ 0,\quad n\to\infty.$$ According to Proposition \ref{lem:number visits multidim_log} and using the fact that the set $A$ contains finitely many points with integer coordinates, it is enough to show that, for any $b>0$ and all $x\in A$,
\[
\sum_{k=1}^{\lfloor b\log n\rfloor } \frac{|\eta^{(x)}_k|}{\sqrt{n}}~\toP~ 0,\quad n\to\infty.
\]
Observe that
\[
\sum_{k=1}^{\lfloor b \log n\rfloor } |\eta^{(x)}_k| \leq \lfloor b \log n\rfloor \max_{1\leq k\leq \lfloor b \log n\rfloor}\,|\eta^{(x)}_k|\quad \text{a.s.}
\]
and that, for all $\ve>0$ and all $x\in A$,
\begin{multline*}
\Pb\Big\{\frac{\log n}{\sqrt{n}}\max_{1\leq k\leq \lfloor b \log n\rfloor}\,|\eta^{(x)}_k|\leq \ve\Big\}\Big(\Pb\Big\{|\eta^{(x)}|\leq \frac{\ve \sqrt{n}}{\log n}\Big\}\Big)^{\lfloor b \log n\rfloor}\\= \Big(1-o(1/\log(\frac{\ve \sqrt{n}}{\log n}))\Big)^{\lfloor b\log n\rfloor}=\big(1+o(1/\log n)\big)^{\lfloor b \log n\rfloor}~\to~ 1,\quad n\to\infty.
\end{multline*}
Here, we have used \eqref{eq:35_62} for the second equality.

The proof of Theorem \ref{thm:3_4_limit_transient1} is complete.
\end{proof}

\section{Proof of Theorem \ref{thm:negative result}}

Put $l_1(t):=\log t$ and $l_2(t):=\log\log t$. Let $\eta_1$, $\eta_2,\ldots$ be independent copies of a random variable $\eta$ with distribution $\Pb\{\eta=x\}= \Pb\{X(1)=x\}$ for $x\in\mbZ^2$. For each fixed $x\in \mathbb{Z}^2\backslash\{0\}$,
denote by $\tau(x)$ the first hitting time of $0$ by a simple symmetric random walk in $\mathbb{Z}^2$ which starts at $x$ and is independent of $\eta$.  Let $\tau_1(x)$, $\tau_2(x),\ldots$ be independent copies of $\tau(x)$, which are also independent of $\eta_1$, $\eta_2,\ldots$.

To prove \eqref{eq:2_4_conv_abs_infty} it is sufficient to show that
\bel{eq:2_4_conv_infty_to_prove}
\liminf_{n\to\infty}\Pb\Big\{\bigcup_{k=1}^{l_2(n)}\Big\{\max_{1\leq j\leq k-1}\,|\eta_j|\leq \eee^{\sqrt{l_1(n)}},\ |\eta_k|>2n,\ \ \sum_{j=1}^{k-1}\tau_j(\eta_j)\leq n\Big\}\Big\}>0.
\ee
Observe that the event in \eqref{eq:2_4_conv_infty_to_prove} coincides with the event
\begin{multline*}
\{\text{there exists}~ k\leq l_2(n)~\text{such that the}~ k\text{th jump from}~ 0~\text{occurs before time}~n~\\\text{and is larger than}~2n,\text{and all the previous jumps do not exceed}~ \eee^{\sqrt{l_1(n)}}\}.
\end{multline*}
Here and hereafter, to simplify notation we do not write the integer parts in summation or union ranges.

We bound the latter probability from below by
\begin{multline}\label{eq:est_983}
\Pb\Big\{\bigcup_{k=1}^{l_2(n)}\Big\{\max_{1\leq j\leq k-1}\,|\eta_j|\leq \eee^{\sqrt{l_1(n)}},\ |\eta_k|>2n\Big\}\Big\}\\-
\Pb\Big\{\bigcup_{k=1}^{l_2(n)}\Big\{\max_{1\leq j\leq k-1}\,|\eta_j|\leq \eee^{\sqrt{l_1(n)}},\ \ \sum_{j=1}^{k-1}\tau_j(\eta_j)> n\Big\}\Big\}
\end{multline}
The first term is equal to
\begin{multline*}
\Pb\Big\{\bigcup_{k=1}^{l_2(n)}\Big\{\max_{1\leq j\leq k-1}\,|\eta_j|\leq \eee^{\sqrt{l_1(n)}},\ |\eta_k|>2n\Big\}\Big\}=\sum_{k=1}^{l_2(n)}\Pb\Big\{\max_{1\leq j\leq k-1}\,|\eta_j|\leq \eee^{\sqrt{l_1(n)}},\ |\eta_k|>2n\Big\}\\=\sum_{k=1}^{l_2(n)}\Pb\Big\{ \max_{1\leq j\leq k-1}\,|\eta_j|\leq \eee^{\sqrt{l_1(n)}}  \Big\}\
\Pb\{|\eta_k|>2n\}=\sum_{k=1}^{l_2(n)}\Big(\Pb\Big\{  |\eta |\leq \eee^{\sqrt{l_1(n)}}  \Big\}\Big)^{k-1}\,
\Pb\{|\eta|>2n\}\\=\frac{\Big(1-\big(\Pb\big\{|\eta|\leq \eee^{\sqrt{l_1(n)}}\big\}\big)^{\lfloor l_2(n)  \rfloor}\Big)\Pb\{|\eta|>2n\}}{ \Pb\big\{ |\eta|> \eee^{\sqrt{l_1(n)}}\big\}}.
\end{multline*}
We have used the fact that the events $\Big\{\max_{1\leq j\leq k-1}\,|\eta_j|\leq \eee^{\sqrt{l_1(n)}},\ |\eta_k|>2n\Big\}$ are disjoint for different $k$. This entails that the probability of the corresponding union is equal to the sum of probabilities.

We estimate the second term of \eqref{eq:est_983} with the help of Boole's inequality:
\begin{multline*}
\Pb\Big\{\bigcup_{k=1}^{l_2(n)}\Big\{\max_{1\leq j\leq k-1}\,|\eta_j|\leq \eee^{\sqrt{l_1(n)}},\ \ \sum_{j=1}^{k-1}\tau_j(\eta_j)> n\Big\}\Big\}\\
\leq \sum_{k=1}^{l_2(n)}\Pb\Big\{\max_{1\leq j\leq k-1}\,|\eta_j|\leq \eee^{\sqrt{l_1(n)}},\ \ \sum_{j=1}^{k-1}\tau_j(\eta_j)> n\Big\}.
\end{multline*}
Combining fragments together we infer
\begin{multline*}
\Pb\Big\{\bigcup_{k=1}^{l_2(n)}\Big\{\max_{1\leq j\leq k-1}\,|\eta_j|\leq \eee^{\sqrt{l_1(n)}},\ |\eta_k|>n,\ \ \sum_{j=1}^{k-1}\tau_j(\eta_j)\leq n\Big\}\Big\}\\ \geq
\frac{\Big(1-\big(\Pb\big\{|\eta|\leq \eee^{\sqrt{l_1(n)}}\big\}\big)^{\lfloor l_2(n)  \rfloor}\Big)\Pb\{|\eta|>n\}}{\Pb\{|\eta|> \eee^{\sqrt{l_1(n)}}\}}\\
-\sum_{k=1}^{l_2(n)}\Pb\Big\{\max_{1\leq j\leq k-1}\,|\eta_j|\leq \eee^{\sqrt{l_1(n)}},\ \ \sum_{j=1}^{k-1}\tau_j(\eta_j)> n\Big\}.
\end{multline*}

It follows from $$\Pb\big\{|\eta|>\eee^{\sqrt{l_1(n)}}\big\}~\sim~\frac{a}{l_2(\eee^{\sqrt{l_1(n)}})}=\frac{a}{l_1({\sqrt{l_1(n)}})}=\frac{2a}{l_2 (n)},\quad n\to\infty$$ that the first summand on the right-hand side of the last inequality converges to $(1-\eee^{-2a})/2$ as $n\to\infty$. We intend to prove that the second summand vanishes. As a preparation, we write
\begin{multline*}
\sum_{k=1}^{l_2(n)} \Pb\Big\{\max_{1\leq j\leq k-1}\,|\eta_j|\leq \eee^{\sqrt{l_1(n)}},\ \ \sum_{j=1}^{k-1}\tau_j(\eta_j)> n\Big\}\\
\leq \sum_{k=1}^{l_2(n)} \Pb\Big\{\bigcup_{j=1}^{l_2(n)}\big\{|\eta_j|\leq \eee^{\sqrt{l_1(n)}},\ \ \tau_j(\eta_j)> \sqrt{n}\big\}\Big\}\\\leq (l_2(n))^2\Pb\{\eta\leq \eee^{\sqrt{l_1(n)}}, \ \tau(\eta)>\sqrt{n}\}.
\end{multline*}
According to formulae (2.16) and (2.17) in \cite{Erdos+Taylor:1960}, for any $x\in\mbZ^2$ satisfying $|x|\leq a_n=O(n^{1/3})$ there exists a constant $B>0$ such that
\[\Pb\{\tau(x)>n\}\leq B l_1(a_n)/l_1(n).\]
This entails
\begin{multline*}
(l_2(n))^2\Pb\big\{|\eta|\leq \eee^{\sqrt{l_1(n)}}, \ \tau(\eta)>\sqrt{n}\big\}=(l_2(n))^2\sum_{x\in\mathbb{Z}^2: |x|\leq \eee^{\sqrt{l_1(n)}} }\Pb\{\eta=x\}\Pb\{\tau(x)>\sqrt{n}\}\\
\leq (l_2(n))^2 \frac{Bl_1(\eee^{\sqrt{l_1(n)}})}{l_1(\sqrt{n})}=\frac{2B (l_2(n))^2}{\sqrt{l_1(n)}} \to 0,\quad n\to\infty,
\end{multline*}
thereby completing the proof of \eqref{eq:2_4_conv_infty_to_prove}. The last claim of the theorem concerning tightness follows from the fact that the supremum functional is continuous in all the four Skorokhod topologies.

\bigskip
\noindent {\sc Acknowledgement}. The research was supported by the High Level Talent Project DL2022174005L of Ministry of Science and Technology of PRC.

\end{document}